\documentclass[11pt,letterpaper]{amsart}

\usepackage[margin=1.4in]{geometry}

\usepackage[pdftex]{graphicx, hyperref}
\usepackage{amsmath, amssymb, amsthm, float}
\usepackage{cancel}
\usepackage{tikz}
\usepackage{lipsum} 
\usepackage[dvipsnames]{xcolor}
\usepackage{datetime} 
\usepackage{svg}
\usepackage{caption}
\usepackage{tabularray}
\usepackage[all]{xy}
\usepackage{quiver}
\usepackage{needspace}
\usepackage[pagewise]{lineno}

\usepackage{todonotes}
  
\numberwithin{equation}{section}

\newcommand{\PLN}{\mathrm{PL}_{n}(\mathbb{R})}

\theoremstyle{plain}  
\newtheorem{thm}{Theorem}[section]
\newtheorem{prop}[thm]{Proposition}
\newtheorem{cor}[thm]{Corollary}
\newtheorem{lem}[thm]{Lemma}

\theoremstyle{definition}
\newtheorem{defn}[thm]{Definition}

\newtheorem{term}[thm]{Terminology}
\newtheorem{exa}[thm]{Example}

\theoremstyle{remark}

\begin{document}

\title[Forest diagrams and lengths for the generalised Thompson's group $F(n)$]{Forest diagrams and lengths for the generalised Thompson's group $F(n)$}

\date{\today}


\author[M.~Gómez Reynolds]{Martín Gómez Reynolds}
\address{University of Glasgow, Glasgow, Scotland}
\email{martinantonio2002@gmail.com}

\begin{abstract}
 We extend the concept of \emph{two-way forest diagrams}, introduced by Belk and Brown in 2003, to represent elements of $F(n)$ as a pair of infinite, bounded $n$-ary forests together with an order-preserving bijection of the leaves. This representation allows us to develop an alternative way to compute the length of an element of $F(n)$, distinct from the formula established by Fordham and Cleary in 2009. As an application of our length formula, we re-prove the existence of \emph{dead end elements} in $F(n)$ and show that their depth is always two,  first proved by Wladis in 2009. 
\end{abstract}

\maketitle

\section{Introduction}

Thompson's group $F$ is one of the most well-studied groups in geometric group theory. As an extension of the work done by Higman in \cite{higman1974finitely}, Brown defined a generalisation of $F$ in \cite{BROWN198745}, called the generalised Thompson group $F(n)$ or the Brown-Thompson group $F(n)$. This group shares many of the interesting properties that $F$ has. It can be defined by the following infinite presentation: 
\begin{align*}
    F(n) \cong \langle x_0, x_1, x_2, \dots \mid x_j x_i  = x_{i} x_{j+n-1}  \textit{ for } i < j\rangle, 
\end{align*}

Fordham's method for computing lengths of elements in $F$ with respect to the $\{x_0, x_1\}$ generating set, presented in \cite{blake2003minimal}, enabled results such as the existence of dead end elements in $F$ and the fact that $F$ is not almost convex, shown by Cleary and Taback in \cite{cleary2004combinatorial} and \cite{cleary2003thompson}, respectively. Later, Belk and Brown in \cite{belk2005forest} developed an alternative length formula for $F$ using forest diagrams, which can be viewed as a simplification of Fordham's work. Forest diagrams have subsequently proven useful in the study of the Cayley graph of $F$ \cite{BelkBux2005,francesco,elder2010counting}. In 2009, Fordham and Cleary extended Fordham's original work to provide a length formula for the group $F(n)$ with respect to the $\{ x_0, x_1, \dots, x_{n-1}\}$ generating set. 

The main result  of this paper consists in adapting Belk and Brown's forest diagrams to represent elements in $F(n)$, thus providing an alternative to Fordham and Cleary's method to calculate lengths in $F(n)$. Increasing the value of $n$ introduces significant complications in the way that forest diagrams are constructed. In particular, the homeomorphism $\psi_n$ from Proposition \ref{homeomorphism}, required to prove that $F(n)$ and $\PLN$ are isomorphic, is notably more complex than the one defined by Burillo in \cite[Section 1.4]{burrillobook} for $n=2$. Also, the larger values of $n$ impose restrictions on how the generators of $F(n)$ interact with forest diagrams: left-multiplying by $x_0$ moves us along the forest diagram in steps of $n-1$; left-multiplying by the remaining $n-1$ generators constructs \textit{$n$-carets} on the \textit{current tree} and the $n-2$ trees immediately to its right. This is one of the major differences to forest diagrams in $F(2)$ and is examined in depth in the latter half of Section \ref{section1}. In Section \ref{section2}, we give a classification of carets and spaces of forest diagrams of elements of $F(n)$. This classification is the basis for the labelling scheme presented in Section \ref{thelengthformulasection}, leading us to the generalised length formula in Theorem \ref{genthm}. Section \ref{proof-section} is entirely devoted to proving Theorem \ref{genthm}. The proof is quite long but relatively  straightforward, simply examining all the possible ways the generators of $F(n)$ can affect the length of an element. In Section \ref{final-section}, we conclude with an application of our new formula to simplify Wladis' work on dead end elements in $F(n)$, presented in \cite{wladis2008unusual}.

This paper is based on the author’s Master’s thesis at the University of Glasgow. I would like to thank my thesis supervisor, Jim Belk, for his support and advice throughout the last few months. His commitment to teaching and mathematics is an inspiration and has made this project a joy to work on.

\section{Forest Diagrams for \texorpdfstring{$F(n)$}{F(n)}} \label{section1}

\begin{defn}
 For $n\in \mathbb{N} \text{ and } n\geq 2 $, let $\PLN$ be the group of all piecewise-linear, orientation-preserving homeomorphisms $f$ from $\mathbb{R}$ to itself such that 

   \begin{enumerate}
        \item $f$ has finitely many breakpoints in the ring $\mathbb{Z}  \left[ \frac{1}{n} \right] $,
        \item on each interval of differentiability the slope of $f$ is a power of $n$, and
        \item  at the rightmost and leftmost interval of differentiability the slope of $f$ is 1  and the constant term is a multiple of $n-1$.
    \end{enumerate}\textbf{}

\end{defn}

\begin{prop} \label{homeomorphism}
    $\PLN$ and $F(n)$ are isomorphic. 
\end{prop}
\begin{proof}
Let $n\in \mathbb{N} \text{ and }  n\geq 2 $. Define $\psi_n: \mathbb{R} \rightarrow [0,1]$ to be the piecewise-linear homeomorphism with breakpoints at
\[
\psi_n(k) =
\begin{cases} 
\displaystyle 
\frac{1}{n^{\lfloor \frac{-k}{n-1} \rfloor + 1}} - \frac{-k \bmod (n-1)}{n^{\lfloor \frac{-k}{n-1} \rfloor + 2}}, 
& \text{if } k \in \mathbb{Z} \text{ and } k < 0, \\[12pt]
\displaystyle 
1 - \frac{1}{n^{\lfloor \frac{k - n + 2}{n-1} \rfloor + 1}} 
+ \frac{(k - n + 2) \bmod (n-1)}{n^{\lfloor \frac{k - n + 2}{n-1} \rfloor + 2}}, 
& \text{if } k \in \mathbb{Z} \text{ and } k \geq 0.
\end{cases}
\]

 Thus, the desired isomorphism between $F(n)$ and $\PLN$ for any $f\in F(n)$ is $f \mapsto \psi_n  f  \psi_n^{-1}$ .
\end{proof}

\begin{defn}
Let $k, p, n \in \mathbb{N}$. We call an interval of the form 
 \begin{equation*}
     \left[ \frac{k}{n^p},\frac{k+1}{n^p}\right]
 \end{equation*} 
an \emph{$n$-adic interval}. Note that the fractions need not be written in reduced form. Partitions of an interval $I$ into $n$-adic subintervals are called \emph{$n$-adic subdivisions} of $I$.

\end{defn}

\enlargethispage{2\baselineskip}

 Visually, we can think of $\psi_n$ as linearly sending the intervals $[k,k+1]$, where $k\in \mathbb{Z}$, in an order-preserving way onto the infinite $n$-adic subdivision of $[0,1]$ represented by the infinite $n$-ary tree pictured in Figure \ref{fig:homeomorphism_tree}, with the interval $[0,1]$ being sent to the $n$-adic interval $\left[ \frac{1}{n}, \frac{2}{n}\right]$, except for the $n=2$ case where $[0,1]$  gets sent to $\left[ \frac{1}{2}, \frac{3}{4}\right]$.

 \newpage
 
\begin{figure}[H]
    \centering
    \includegraphics[width=0.6\linewidth]{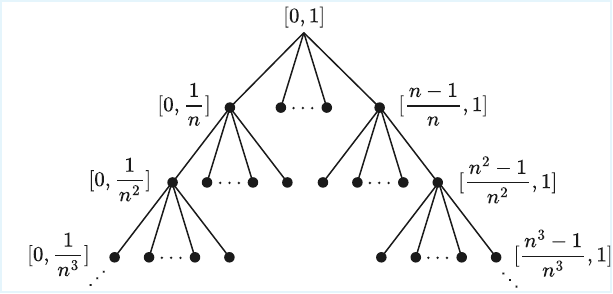}
    \vspace{-2mm}
    \caption{Infinite $n$-ary tree.}
    \label{fig:homeomorphism_tree}
\end{figure}

Note that the homeomorphism $\psi_n$ is congruent with the one defined by Burillo in \cite[Section 1.4]{burrillobook} for the case $n=2$.

Consider the real line as being subdivided by the integers, as shown below. 
\vspace{5mm}
\begin{center}
\begin{tikzpicture}[scale=1.5]
    \draw[latex-latex] (-3.5,0) -- (3.5,0) node[right] {};

    \foreach \x in {-3,-2,-1,0,1,2,3} {
        \draw (\x,0.1) -- (\x,-0.1); 
        \node[below] at (\x, -0.1) {\x}; 
    }
\end{tikzpicture}
\end{center}
Define an \emph{$n$-adic subdivision} of $\mathbb{R}$ as the result of subdividing finitely many of these intervals, of the form $[k,k+1]$ where $k \in \mathbb{Z}$, into $n$-adic subintervals.

We can represent any $n$-adic subdivision of $\mathbb{R}$ with a finite collection of non-trivial $n$-ary trees, which we will call an \textit{$n$-ary forest}. 

\begin{exa} \label{exa4.11}
    Two examples of a $3$-adic subdivision of $\mathbb{R}$ and their corresponding ternary forest. We denote the tree that represents the subdivision of the unit interval with a pointer.

\begin{figure}[H]
    \centering
    \hspace{-11mm}
    \includegraphics[width=0.92\linewidth]{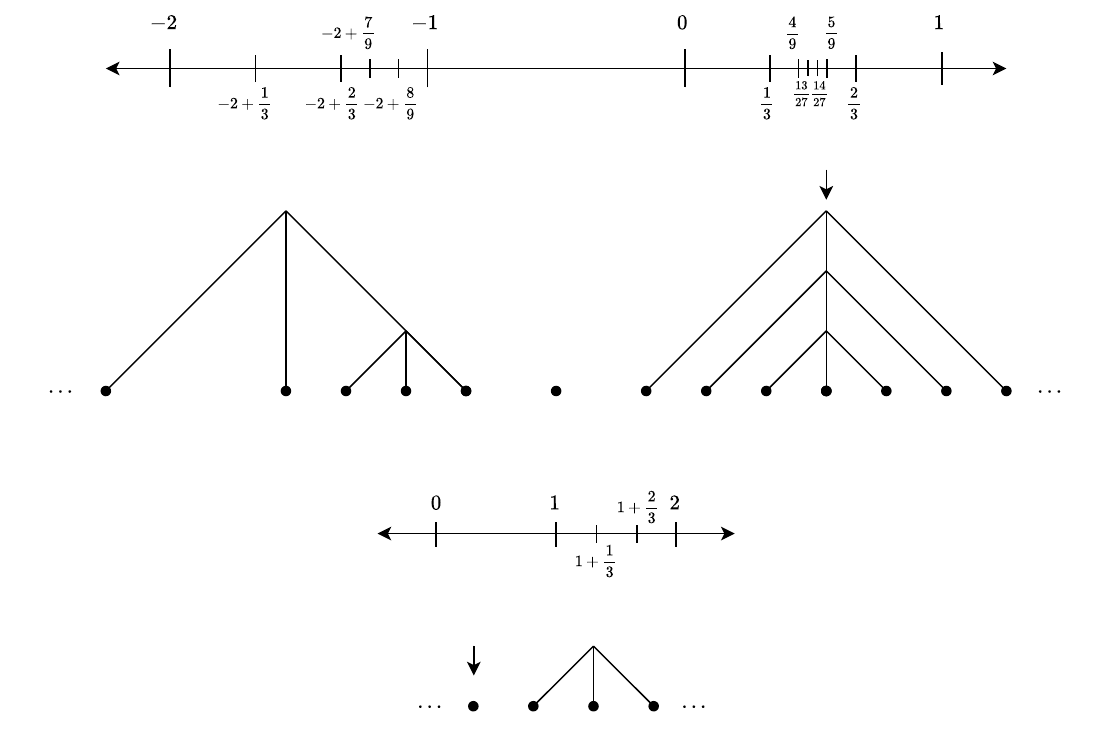}

\end{figure}

Each tree corresponds to an interval $[k,k+1]$ where $k \in \mathbb{Z}$. Each leaf corresponds to a 3-adic subdivision of that interval. Note that there are infinitely many trivial trees for all of the intervals that do not get subdivided.

\end{exa}

\begin{prop} \label{analogousthm}
    Let $f\in \PLN$. There exist $n$-adic subdivisions $\mathcal{D}$ and $\mathcal{R}$ of $\mathbb{R}$ such that $f$ maps each interval of $\mathcal{D}$ linearly onto an interval of $\mathcal{R}$.
\end{prop}

Thus, we can represent any $f \in \PLN$ by a pair of $n$-ary forests, together with an order-preserving bijection of their leaves. We call this an \emph{$n$-ary forest diagram} for $f$. We inherit the same convention as used in \cite{belk2005forest}: the forest representing the domain of $f$ will be on top, and the forest representing the range of $f$ will be below. 

\begin{exa}
    Let  $f \in \PLN$ be the homeomorphism that maps the two $3$-adic subdivisions of $\mathbb{R}$ given in Example \ref{exa4.11} to each other, with the interval $\left[ -1, 0\right]$ being mapped to itself. The forest diagram corresponding to such an $f$  is shown in the following figure.

    \begin{figure}[H]
        \centering
        \includegraphics[width=0.8\linewidth]{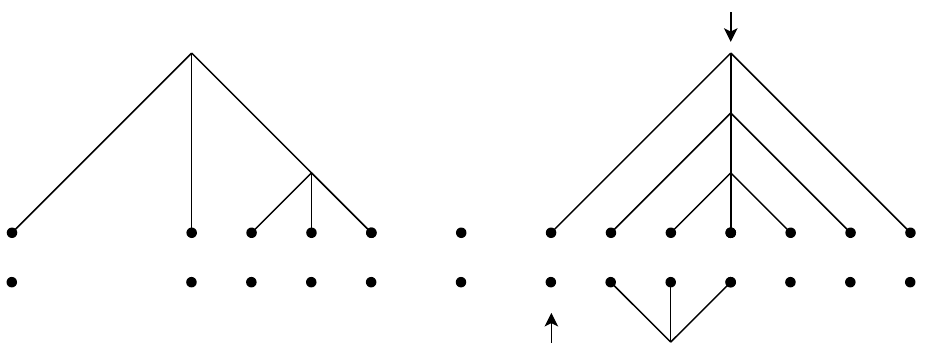}
    \end{figure}
\end{exa}
\enlargethispage{5\baselineskip}
A good reason to define forest diagrams for $F(n)$ is how nicely they interact with the generating set $\{x_0, x_1, \dots, x_{n-1}\}$ of $F(n)$. The following figure shows the forest diagrams of the $n$ generators of $F(n)$.
\vspace{-3mm}
\begin{figure}[H]
    \centering
    \includegraphics[width=0.95\linewidth]{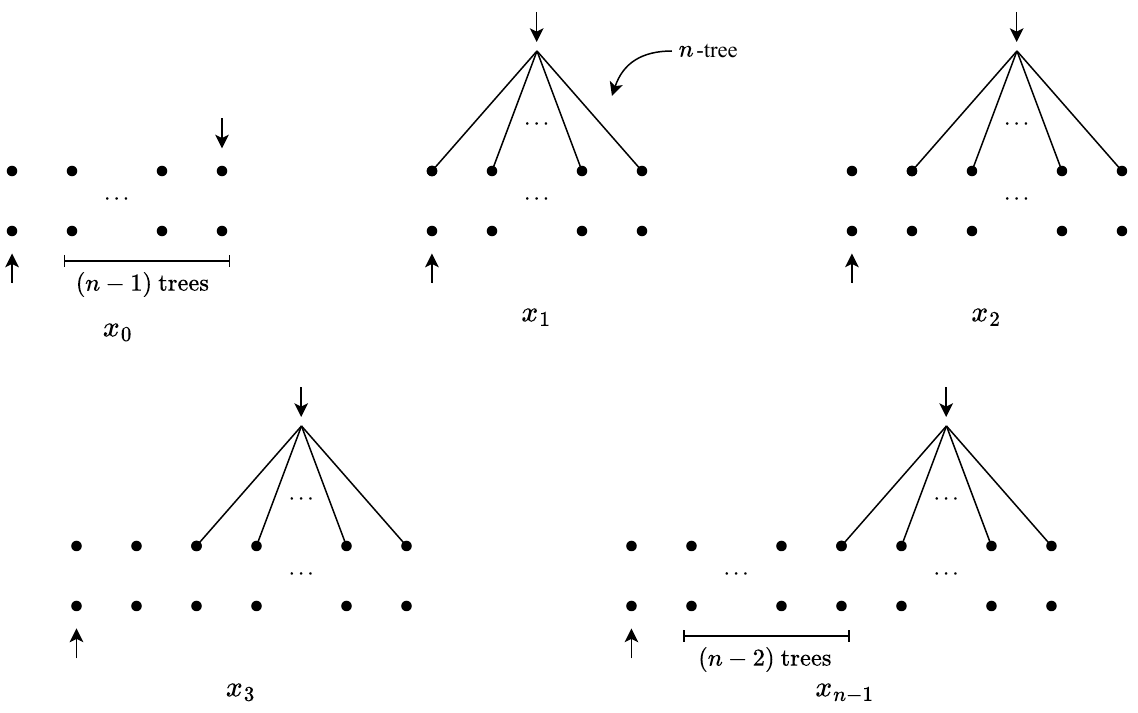}
    \caption{Forest diagrams of the $n$ generators of $F(n)$.}
\end{figure}
\vspace{-3mm}

\newpage

\begin{defn}
    An \emph{n-caret} of a rooted $n$-ary tree is the subgraph containing a node (or the root), the $n$ edges going down from it, and the $n$ vertices at their ends (the node's children).
\end{defn}
\vspace{-3mm}
\begin{figure}[H]
    \centering
    \includegraphics[width=0.3\linewidth]{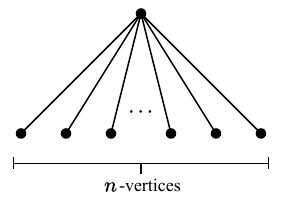}
    \caption{An $n$-caret.}
\end{figure}

Let us now examine how left-multiplying by the finite generating set of $F(n)$ affects forest diagrams. The following proposition can be proven by direct computation. Note that we refer to the tree directly underneath the top pointer as the \emph{current} tree.

\begin{prop}
    Let $\mathfrak{f}$ be a forest diagram for some $f \in F(n)$. Then:

    \begin{enumerate}
        \item A forest diagram for $x_0f$ can be obtained by moving the top pointer of $\mathfrak{f}$ $n-1$ trees to the right. 

        \item A forest diagram for $x_if$ for any $i\in \{1,\dots,n-1\}$ can be obtained by attaching an $n$-caret to the roots of all trees between the $(i-1)$-th tree and $(i+n-1)$-th tree to the right of the current tree, in the top forest of $\mathfrak{f}$. The top pointer will not move, except for the $i=1$ case, where it will point to the root of the attached caret. 
    \end{enumerate}
\end{prop}  

\begin{exa}
Let $f \in F(4)$ have the following forest diagram. 
    \begin{figure}[H]
        \centering
        \includegraphics[width=1\linewidth]{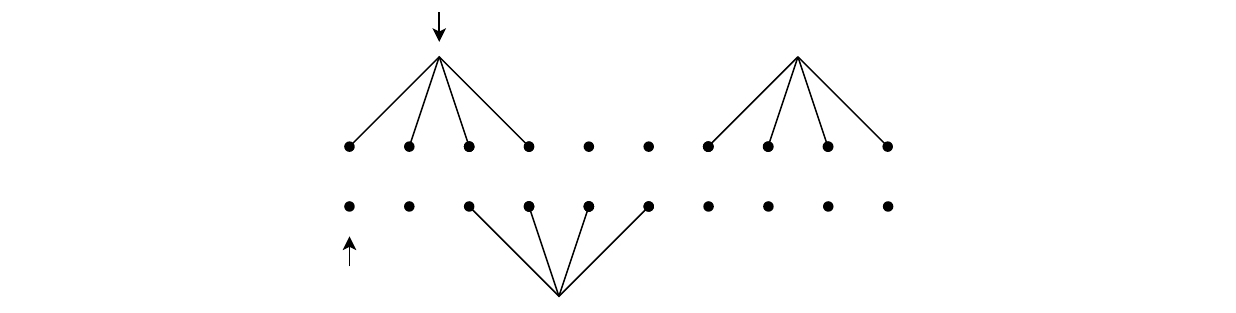}
    \end{figure}

Then the forest diagrams of $x_if$ for $i\in \{0,1,2,3\}$ are pictured in Figure \ref{fig:forest_diagram_example}.

    \begin{figure}[H]
        \centering
        \includegraphics[width=1\linewidth]{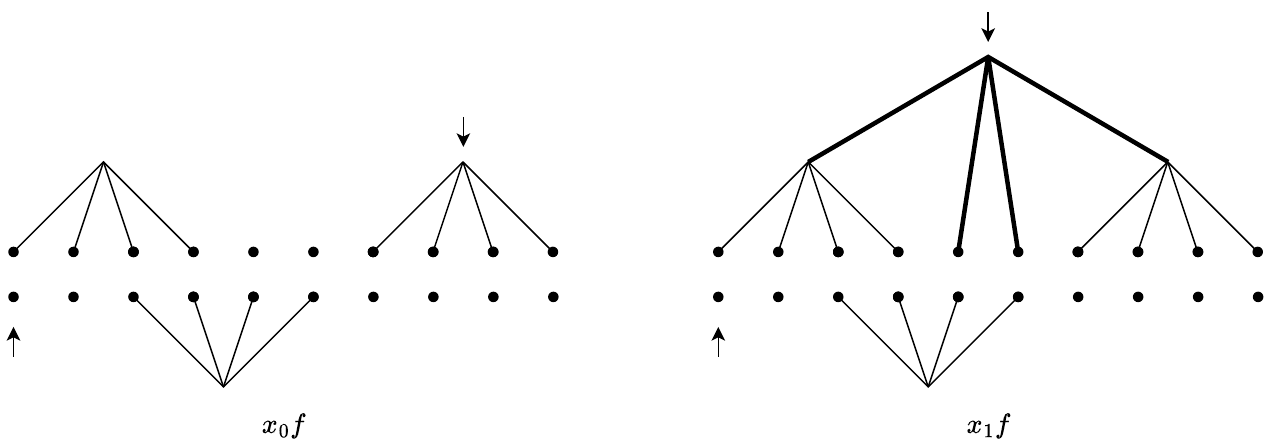}
    \end{figure}
    \begin{figure}[H]
        \centering
        \includegraphics[width=1\linewidth]{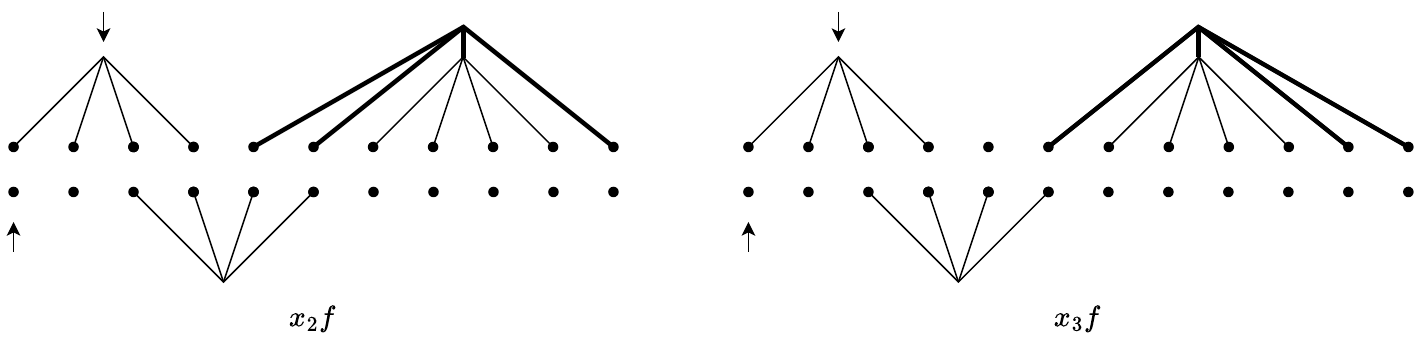}
        \caption{Forest diagrams of $x_if$ for $i\in \{0,1,2,3\}$.}
        \label{fig:forest_diagram_example}
    \end{figure}
    So left-multiplying by $x_0$ moves the top pointer through the forest diagram, whereas $x_1,x_2$ and $x_3$ build carets to the right of the top pointer. 
\end{exa}

We say a forest diagram is \textit{reduced} if it has no opposing pairs of $n$-carets. Reduction of $n$-ary forest diagrams works in the same way as the binary case.  In particular, if $\mathfrak{f}$ is a reduced forest diagram of $f\in F(n)$ then the forest diagram of $x_0f$ will also be reduced. However, the forest diagram of $x_i f$ for $i \in \{1,\dots,n-1\}$ need not be reduced. Also, as in the binary case, we have the following proposition. 

\begin{prop}
    Any element of $\PLN$ has a unique reduced forest diagram representing it. 
\end{prop}

To end this section, let us consider how left-multiplying by the inverse of the generators of $F(n)$ interacts with forest diagrams. 

\begin{prop}
    Let $\mathfrak{f}$ be a forest diagram for some $f \in F(n)$. Then:

    \begin{enumerate}
        \item A forest diagram for $x^{-1}_0f$ can be obtained by moving the top pointer of $\mathfrak{f}$ $n-1$ trees to the left. 

        \item A forest diagram for $x^{-1}_if$ for any $i\in \{1,\dots,n-1\}$ can be obtained by ``dropping a negative caret" in the $(i-1)$-th tree. In particular:
        \begin{itemize}
            \item If the $(i-1)$-th tree to the right of the current tree is non-trivial, then left-multiplying by $x^{-1}_i$ removes the top caret of the $(i-1)$-th tree and the top pointer does not move, except for the $i=1$ case where it moves to the leftmost child of the original current tree.
            \item If the $(i-1)$-th tree to the right of the current tree is trivial, then left-multiplying by $x_i^{-1}$  attaches a caret in the bottom forest of $\mathfrak{f}$ to the leaf beneath the  $(i-1)$-th tree and creates $n-1$ trivial trees in the top forest of $\mathfrak{f}$ immediately to the right of the $(i-1)$-th tree. I.e.\ the negative caret ``falls through" to the bottom forest.
        \end{itemize}
    \end{enumerate}
\end{prop}

\begin{exa}
    Let $f\in F(3)$ have the following forest diagram. 
\begin{figure}[H]
        \centering
        \includegraphics[width=1\linewidth]{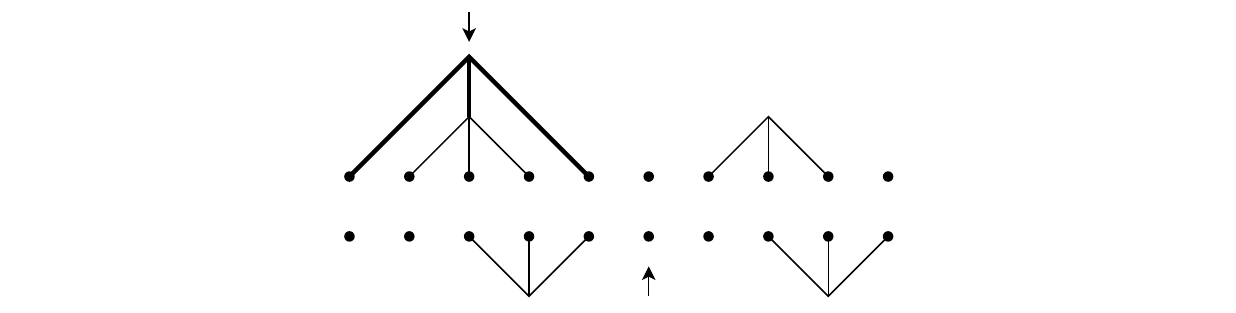}
    \end{figure}
Then the forest diagrams of $x^{-1}_1f$ and $x^{-1}_2f$ are the following. 

        \begin{figure}[H]
        \centering
        \includegraphics[width=1\linewidth]{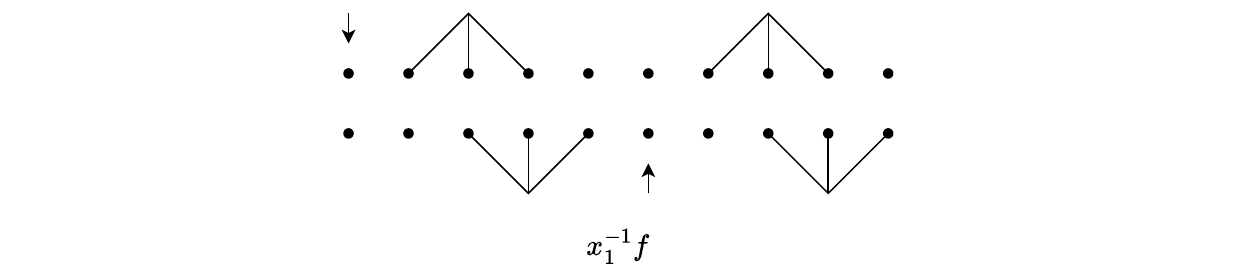}
    \end{figure}
\vspace{-1cm}
        \begin{figure}[H]
        \centering
        \includegraphics[width=1\linewidth]{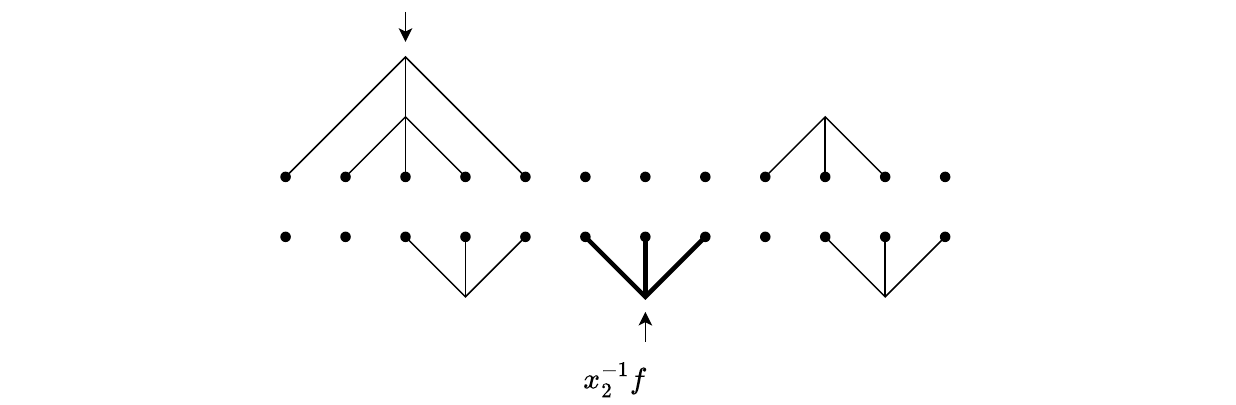}
    \end{figure}

    Left-multiplying $f$ by $x_1^{-1}$ simply removed the highlighted caret. However, as there was no caret in the top forest for $x_2^{-1}$ to remove, a new (highlighted) caret was added to the bottom forest. This results in the creation of two new spaces, one tree to the right of the current tree. 
\end{exa}

\section{Classifying Carets and Spaces}
\label{section2}

Let us number the spaces of a forest diagram $\mathfrak{f}$ for $f\in F(n)$ as follows: the space immediately to the left of the top pointer will be numbered as $0$. The spaces to the left of $0$ will be $-1, -2, -3, \dots$, moving further left; and the spaces to the right of $0$ will be numbered as $1, 2, 3, \dots$, moving further right. 

\begin{defn}
A \emph{marked space} is a space in $\mathfrak{f}$ that is numbered by a multiple of $n-1$.
\end{defn}

Note that, for any $f \in F(2)$, all spaces in the forest diagram of $f$ will be marked spaces. 

\begin{defn}
    We call a space in an $n$-ary forest \emph{interior} if it lies under a tree (or over a tree, if the forest is upside-down) and \emph{exterior} if it lies between two trees.
\end{defn}

\begin{defn}
    If the root of an $n$-caret is attached to the leaf of another $n$-caret, we call the former a \emph{child} and the latter its \emph{parent}. Note that a parent caret can have up to $n$ children.
\end{defn}

\begin{defn}
    Let $\mathfrak{f}$ be an $n$-ary forest. We say an $n$-caret in $\mathfrak{f}$ \emph{encloses} the $n-1$ spaces that lie immediately underneath that caret. 
\end{defn}

\begin{exa}    
The $n-1$ spaces enclosed by each $n$-caret in the forest diagram of $f$ from Example \ref{exaaa} are highlighted with colours in the following figure.
\begin{figure}[H]
    \centering
    \includegraphics[width=1\linewidth]{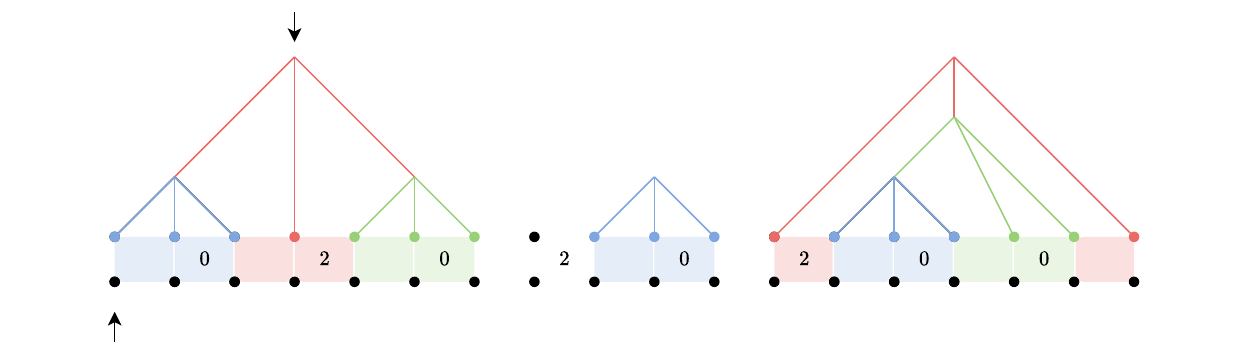}
    \caption{Correspondence between spaces and their enclosing $n$-carets.}

\end{figure}
\end{exa}

\begin{exa} \label{exaaa}
    Consider the element $f\in F(3)$ with the forest diagram from the previous figure. Let us construct, in the most efficient way possible, the forest diagram of $f$ by left-multiplying with the generating set $\{x_0,x_1,x_2\}$. Let us note the number of times we must cross each marked space. We first construct the leftmost tree. To do this, we must first build the two child carets to then be able to construct the parent. 
\vspace{-1mm}
    \begin{figure}[H]
        \centering
        \includegraphics[width=\linewidth]{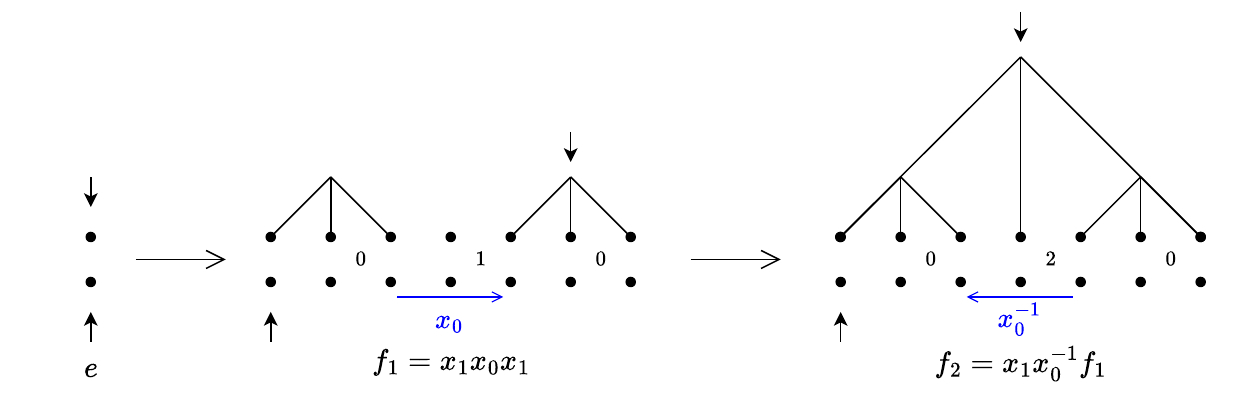}
    \end{figure}

    Notice how we had to cross twice the interior marked space labelled \textbf{2}. Now, we left-multiply by $x^2_1x_0x_1x_0$ to construct the middle tree and start building the rightmost tree. This is pictured in the following figure. 

    \begin{figure}[H]
        \centering
        \includegraphics[width=\linewidth]{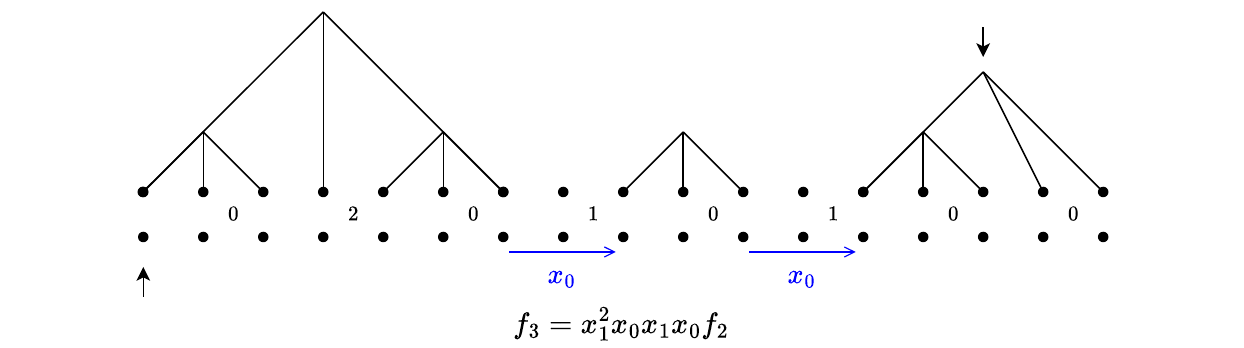}
    \end{figure}

Notice how when building the rightmost tree of the forest diagram of $f_3$, we did not need to cross the two rightmost interior marked spaces. Finally, we move back left (i.e.\ left-multiply by $x^{-1}_0)$ to then complete building the right tree by left-multiplying by $x_2$ and return to the leftmost tree by left-multiplying by $x^{-1}_0$.

\begin{figure}[H]
    \centering
    \includegraphics[width=\linewidth]{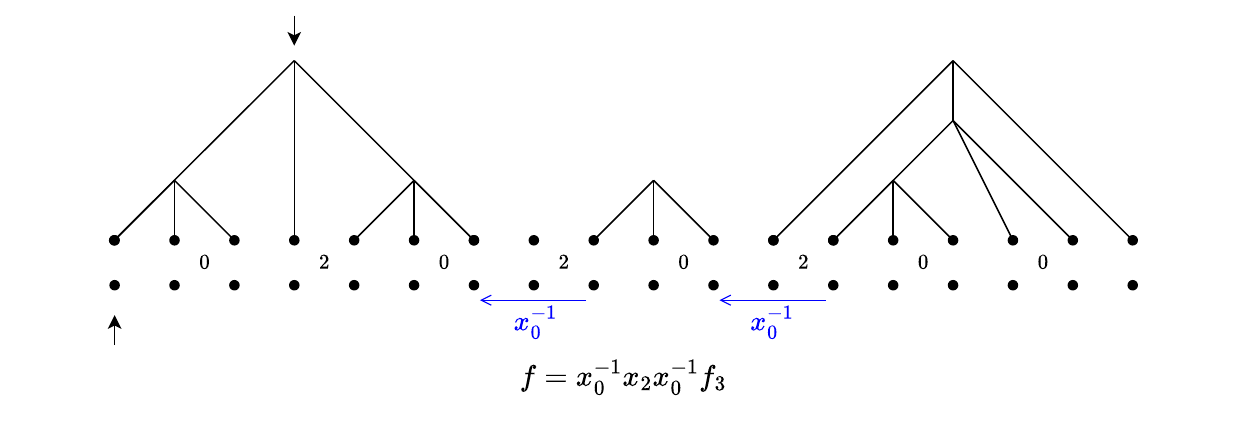}
    \vspace{-10mm}
\end{figure}

Therefore, we get that a minimum-length word for $f$ is \begin{align*}
    x_0^{-1} x_2 x_0^{-1} x_1^{2} x_0 x_1 x_0 x_1 x_0^{-1} x_1 x_0 x_1
\end{align*}
and so $f$ has length 13.
\end{exa}

\section{The Length Formula}  \label{thelengthformulasection}
We are now ready to present the process of finding the length of any element \mbox{$f \in F(n)$}. This will consist of three steps: labelling the marked spaces of the forest diagram of $f$; assigning a weight to each of the marked spaces, according to the combination of labels; and, finally, computing the length of the element with respect to the calculated weights.

Let $f \in F(n)$ and let $\mathfrak{f}$ be its reduced forest diagram. We label the marked spaces of each forest of $\mathfrak{f}$ as follows:

\begin{enumerate}
    \item Label a marked space \textbf{L} (for \emph{left}) if it is exterior and to the left of a pointer. 
    \item Label a marked space \textbf{N} (for \emph{necessary}) if it is exterior and the rightmost marked space to the left of a caret (and not already labelled \textbf{L}); or if it is interior and to the left of a child of the caret that encloses that marked space. 
    \item Label a marked space \textbf{R} (for \emph{right}) if it is exterior and to the right of a pointer (and not already labelled \textbf{N}).
    \item Label a marked space \textbf{I} if it is interior (and not already labelled \textbf{N}).
\end{enumerate}

For the sake of clarity, let us further explain when to label an interior marked space as \textbf{N}. Firstly, observe there is a one-to-one correspondence between interior marked spaces and the carets that enclose them. Each caret has $n$ leaves, which can be numbered left to right from $1$ to $n$.  Let the marked space associated to the caret be in-between the $i$-th and $(i+1)$-st leaf, where $1 \leq i \leq n-1$.  If the caret has a child in the $k$-th leaf and $k > i $ then we label the marked space with \textbf{N}, otherwise it will be labelled \textbf{I}. As an example, see the labelled forest diagram of Figure \ref{fig:labe_example}. 

\begin{figure}[H]
    \centering
    \includegraphics[width=0.45\linewidth]{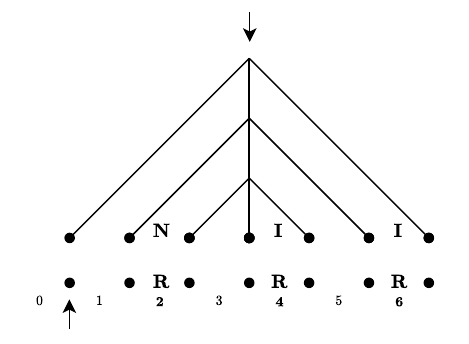}
    \caption{Example of when to label marked spaces as \textbf{N}.}
    \label{fig:labe_example}
\end{figure}

We assign a \emph{weight} to each marked space in the support of $\mathfrak{f}$ according to the combination of its top and bottom labels. See Table \ref{tab:weights}.

\begin{table}[H]
    \centering
    \text{ \hspace{10mm} Top Labels} \\
    \vspace{5pt}
    \begin{tabular}{@{}c@{}}
        \makebox[0pt][c]{
            \rotatebox[origin=c]{90}{{\hspace{-10mm}Bottom Labels}}
        }%
        
        \setlength{\tabcolsep}{8pt} 

    \scalebox{1.1}{ 
        \renewcommand{\arraystretch}{1.5} 
        \begin{tabular}{c|c c c c} 
            & \textbf{L} & \textbf{N} & \textbf{R} & \textbf{I} \\ \hline
            \textbf{L} & 2 & 1 & 1 & 1 \\ 
            \textbf{N} & 1 & 2 & 2 & 2 \\ 
            \textbf{R} & 1 & 2 & 2 & 0 \\ 
            \textbf{I} & 1 & 2 & 0 & 0 \\ 
        \end{tabular}
    }
    \end{tabular}
    \caption{Weight assignment.}
    \label{tab:weights}
\end{table}

\begin{thm}[The generalised length formula]\label{genthm}
 Let $f \in F(n)$ and let $\mathfrak{f}$ be its reduced forest diagram. Then the $\{x_0, \dots, x_{n-1}\}$-length of $f$ is: 

 $$ \ell(f) = \ell_0(f) + \ell_1(f)$$
 where: 
 \begin{enumerate}
     \item $\ell_0(f)$ is the sum of all the weights assigned to the marked spaces of $\mathfrak{f}$, and
     \item $\ell_1(f)$ is the total number of carets in $\mathfrak{f}$.
 \end{enumerate}
\end{thm}

Intuitively, the weight of a marked space can be thought of as the number of times it must be crossed during the construction of the forest diagram. 
\newpage
Let us now look at four examples that cover all label combinations.
\begin{exa}
    Let $f\in F(4)$ be the strongly positive element with the forest diagram shown below. 

    \begin{figure}[H]
        \centering
        \includegraphics[width=0.7\linewidth]{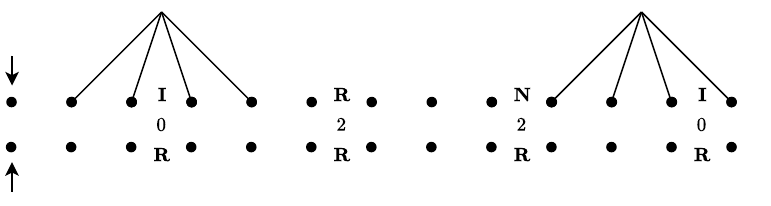}
    \end{figure}

    We have that $\ell_0(f)=4$ and $\ell_1(f)=2$, so the length of $f$ is 6. Indeed, note that 
    $$ x^{-2}_0x_1x_0^2x_2$$
    is a minimum-length word for $f$. 
\end{exa}

\begin{exa}
    Let $f \in F(4)$ be the right-sided element with the following forest diagram.

    \begin{figure}[H]
        \centering
        \includegraphics[width=1\linewidth]{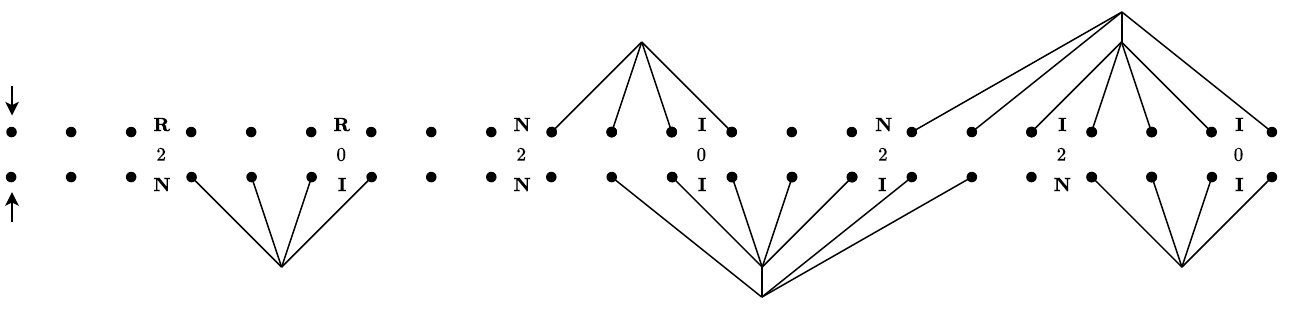}
    \end{figure}
    Notice that $\ell_0(f)=8$ and $\ell_1(f)=7$,  so $\ell(f)=15$. One minimum-length word for $f$ is:
    $$ x_0^{-1} x_1^{-1} x_0^{-2} x_1 x_3 x_0^{-1} x_1^{-1} x_0^2 x_1 x_3^{-1} x_2^{-1} x_0^2 .$$

    Interestingly, note that we were able to avoid crossing the marked space labelled $\rule{0pt}{24pt} \left[ \begin{matrix} \mathbf{R} \\ \mathbf{I} \end{matrix} \right]$ by constructing the caret in the bottom forest later in the process. Thus, marked spaces labelled   $\rule{0pt}{24pt} \left[ \begin{matrix} \mathbf{R} \\ \mathbf{I} \end{matrix} \right]$ carry a weight of 0, as do  $\rule{0pt}{24pt} \left[ \begin{matrix} \mathbf{I} \\ \mathbf{R} \end{matrix} \right]$ and  $\rule{0pt}{24pt} \left[ \begin{matrix} \mathbf{I} \\ \mathbf{I} \end{matrix} \right]$. The same cannot be said for the marked spaces of right-sided elements with any other combination of labels, as all of them need to be crossed twice: once when moving right, away from the pointers, and another time when returning left. 
    
\end{exa}

\newpage

\begin{exa}
    Let $f \in F(5)$ be the left-sided element with the following forest diagram.

    \begin{figure}[H]
        \centering
        \includegraphics[width=1\linewidth]{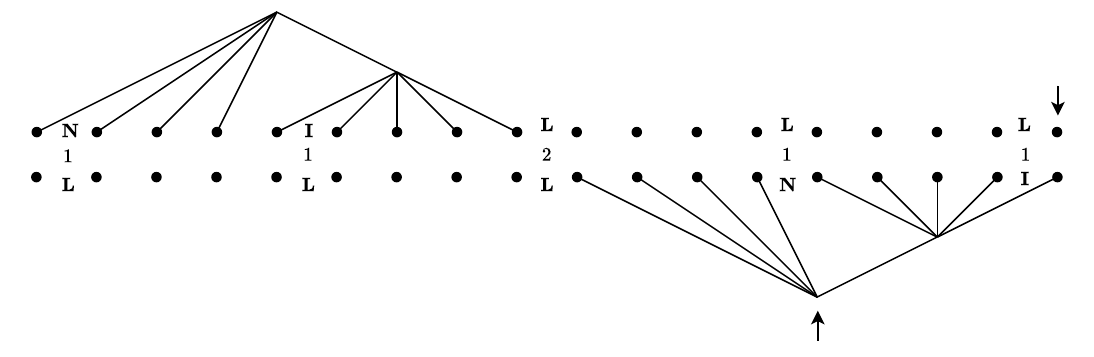}
    \end{figure}

    We have that $\ell_0(f)=6$ and $\ell_1(f)=4$. Therefore, $f$ has length 10. A possible minimum-length word for $f$ is:
    $$ x_0 x_1^{-1} x_0 x_1^{-1} x_0 x_4 x_0^{-1} x_4 x_0^{-2} .$$
    
   In this case, note that all interior marked spaces of $f$ have a weight of 1. This is because $x_i$, for any $i \in \{1, \dots n-1 \} $, builds carets to the right of the top pointer. Therefore, if a marked space is interior, then we must have first crossed that marked space on our way to build the caret that will enclose it. Also, notice how by constructing the non-trivial tree from the bottom forest later in the process, we were able to avoid any extra crossings of the two marked spaces closest to the pointers.
\end{exa}

\begin{exa}
    Let $f \in F(3)$ be the element with the following forest diagram.

    \begin{figure}[H]
        \centering
        \includegraphics[width=1\linewidth]{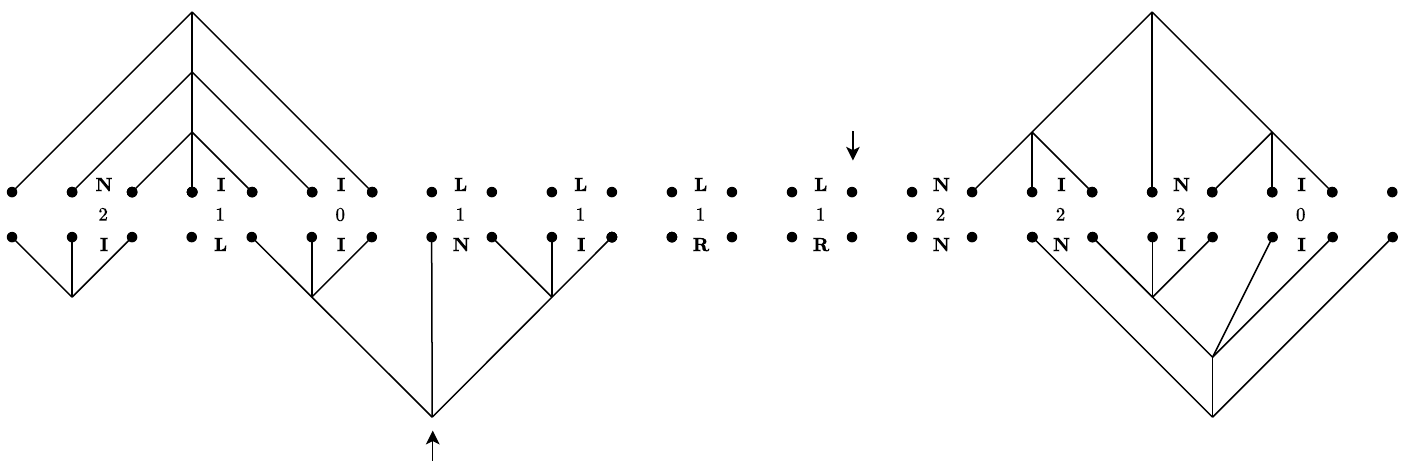}
    \end{figure}

    Then, $\ell_0(f)=13$ and $\ell_1(f)=13$, so $f$ has length $26$. One minimum-length word for $f$ is: 

    $$ x_0^{-1} x_1^{2} x_0^{-2} x_1^{} x_0^{} x_1^{-2} x_0^{} x_2^{-1} x_0^{4} x_1^{-1} x_0^{} x_1^{} x_2^{} x_0^{-1} x_1^{} x_0^{} x_1^{-1} x_0^{-1} x_1^{-2}.$$

    This is the first example with marked spaces labelled $\left[ \begin{matrix} \mathbf{L} \\ \mathbf{R} \end{matrix} \right]$, which we can see only need to be crossed once.  
\end{exa}

\section{Proof of the Length Formula} \label{proof-section}
Proving the generalised length formula from Theorem \ref{genthm} basically follows the same structure as the proof for the case of $F$ found in \cite[Section 4.3]{belk2005forest}. Essentially, our task is to show that the proposed length formula satisfies the conditions of Lemma \ref{lemma} (from \cite[Lemma 2.1]{fordham2009minimal}) and hence the result follows. 
\begin{lem} \label{lemma}
    Let $G$ be a group with generating set $S$, and let $\ell: G \rightarrow \mathbb{N}$ be  a function. Then $\ell$ is the length function for $G$ with respect to $S$ if and only if:
    \begin{enumerate}
        \item $\ell(e)=0$, where $e$ is the identity of $G$.
        \item $|\ell(sg)-\ell(g)| \leq 1$ for all $g \in G$ and $s \in S$.
        \item If $g \in G \setminus \{e\}$, there exists an $s \in S \cup S^{-1}$ such that $\ell(sg) < \ell(g)$.
    \end{enumerate}
\end{lem}

For our purposes, let $\ell$ denote the function defined on $F(n)$ by Theorem \ref{genthm}. It is straightforward to see that $\ell$ satisfies the first condition. To show that $\ell$ satisfies the other two conditions we only need to gather information on how left-multiplication by generators affects the function $\ell$.

\begin{term}
    Let $f \in F(n)$, the \emph{current tree} of $f$ is the tree in the forest diagram of $f$ indicated by the top pointer. The \emph{right space} of $f$ is the leftmost exterior marked space to the right of the current tree, and the \emph{left space} of $f$ is the rightmost exterior marked space to the left of the current tree. 
\end{term}

\begin{prop} \label{prop1}
    If $f \in F(n)$, then $\ell(x_0 f)=\ell(f) \pm 1$. Specifically, $\ell(x_0 f) < \ell(f)$ unless one of the following conditions holds:
    \begin{enumerate}
        \item $x_0 f$ has larger support than $f$.
        \item The right space of $f$ has bottom label \textbf{L}, and left-multiplication by $x_0$ does not remove this space from the support.
        \item The right space of $f$ is labelled $\left[ \begin{matrix} \mathbf{R} \\ \mathbf{I} \end{matrix} \right]$.
    \end{enumerate}
\end{prop}
\begin{proof}
    Clearly, $\ell_1(x_0 f) = \ell_1(f)$. As for $\ell_0$, note that the only change in labelling may come from the right space of $f$. Let us examine the three possible cases: 

    \textit{Case 1:} \quad Suppose $x_0 f$ has larger support than $f$. This implies that the right space of $f$ is originally unlabelled, but then has label $\rule{0pt}{24pt} \left[ \begin{matrix} \mathbf{L} \\ \mathbf{R} \end{matrix} \right]$ in $x_0 f$, which corresponds to a weight of 1. Therefore, $\ell_0(x_0 f)= \ell_0(f)+1$.

    \textit{Case 2:} \quad Suppose  $x_0 f$ has smaller support than $f$. Then the right space of $f$ has label $\left[ \begin{matrix} \mathbf{L} \\ \mathbf{R} \end{matrix} \right]$, but then becomes unlabelled in $x_0 f$. Therefore, $\ell_0(x_0 f)=\ell_0(f)-1$.

    \textit{Case 3:} \quad Suppose $x_0 f$ has the same support as $f$. This means that the right space of $f$ has top label \textbf{N} or \textbf{R}, but top label \textbf{L} in $x_0 f$.  If we look at the relevant rows of the weight table: 
    \begin{table}[H]
    \centering
    \scalebox{1.1}{
     \renewcommand{\arraystretch}{1.3} 
    \begin{tabular}{c| c c c c}  

 & \textbf{L} & \textbf{N} & \textbf{R} & \textbf{I}  \\ \hline
\textbf{L} & 2 & 1 & 1 & 1 \\ 
\textbf{N} & 1 & 2 & 2 & 2 \\ 
\textbf{R} & 1 & 2 & 2 & 0 \\ 

\end{tabular}}
\end{table}
    we notice that each entry of the \textbf{N} and \textbf{R} rows differs from the corresponding entry of the \textbf{L} row by exactly one. Furthermore, moving from an \textbf{N} or \textbf{R} row only increases the weight when in the \textbf{L} column or when starting at $\rule{0pt}{24pt}\left[ \begin{matrix} \mathbf{R} \\ \mathbf{I} \end{matrix} \right]$.
\end{proof}

\begin{cor} \label{cor1}
    Let $f \in F(n)$. Then $\ell(x_0^{-1} f) < \ell(f)$ if and only if one of the following conditions holds:

    \begin{enumerate}
        \item $x_0^{-1} f$ has smaller support than $f$.
        \item The left space of $f$ has label $\left[ \begin{matrix} \mathbf{L} \\ \mathbf{L} \end{matrix} \right]$.     
        \item The left space of $f$ has label $\left[ \begin{matrix} \mathbf{L} \\ \mathbf{I} \end{matrix} \right]$, and the current tree is trivial. 
    \end{enumerate}
\end{cor}

\begin{prop} \label{prop2}
    Let $f \in F(n)$. If left-multiplying $f$ by $x_i$, where $i \in \{1,2,\dots, n-1\}$, cancels a caret from the bottom forest, then $\ell(x_i f) = \ell(f) -1$.
\end{prop}
\begin{proof}
    Clearly, as we are deleting a caret, $\ell_1(x_i f)=\ell_1(f)-1$. Our task is to show that $\ell_0$ remains unchanged. Note that the right space of $f$ is exterior by definition and it can not be the first marked space to the left of a caret (as then there wouldn't be enough ``space'' for $x_i$ to cancel a caret in the bottom forest), thus its top label must be \textbf{R} and with a similar argument we see that its bottom label must be \textbf{I}, so its assigned weight is 0. Furthermore, note that the right space of $f$ will be ``destroyed'' by the deletion of the bottom caret. Thus, $\ell_0$ will be unchanged by this. 

    The only other space that may be affected is the left space of $f$. If this space is not in the support of $f$, it remains unlabelled in $x_i f$. Otherwise, note that it must have top label \textbf{L} in both $f$ and $x_i f$. The relevant row of the weight table is: 

        \begin{table}[H]
    \centering
    \scalebox{1.1}{
     \renewcommand{\arraystretch}{1.3} 
    \begin{tabular}{c| c c c c}  

 & \textbf{L} & \textbf{N} & \textbf{R} & \textbf{I}  \\ \hline
\textbf{L} & 2 & 1 & 1 & 1

\end{tabular}}
\end{table}

Therefore, the only property that could affect $\ell_0$ is whether or not the bottom label is \textbf{L}, which remains unaffected by the deletion of the caret. 

\end{proof}

\begin{cor}\label{hmm}
    Let $f \in F(n)$. If left-multiplying $f$ by $x^{-1}_i$, where $i \in \{1,\dots, n-1\}$, builds a caret in the bottom forest, then $\ell(f) < \ell(x_i^{-1} f)$.
\end{cor}

\begin{prop} \label{prop3}
     Let $f \in F(n)$, and suppose that left-multiplying $f$ by $x_i$, where $i \in \{ 1, \dots, n-1\}$, creates a caret in the top forest. Then $\ell(x_i f) = \ell(f) \pm 1$. Specifically, $\ell(x_i f) < \ell(f) $ if and only if the right space of $f$ has label $\rule{0pt}{24pt}\left[ \begin{matrix} \mathbf{R} \\ \mathbf{R} \end{matrix} \right]$.
\end{prop}
\begin{proof}
    Clearly, as $x_i$ creates a new caret, we have that $\ell_1(x_i f) = \ell_1(f)+1$. Also, note that the right space of $f$ is the only space whose labelling might be affected. Let us examine the two possible cases:

    \textit{Case 1:} \quad Let $x_i f $ have larger support than $f$. This implies that the right space of $f$ is originally unlabelled and so it must be labelled $\rule{0pt}{24pt}\left[ \begin{matrix} \mathbf{I} \\ \mathbf{R} \end{matrix} \right]$ in $x_i f$. 

    \textit{Case 2:} \quad Let $x_i f$ have smaller or equal support than $f$. Then the right space of $f$ has label \textbf{N} or \textbf{R}. 
    \begin{itemize} 
        \item If the top label is \textbf{N}, it will remain an \textbf{N} in $x_i f$ as the right space of $f$ will now be enclosed by the caret created by $x_i$ and to the left of a child of that caret. Hence $\ell_0$ is unaffected.
        \item If the top label is \textbf{R}, then it must change to an \textbf{I} in $x_i f$. Examining the relevant rows of the weight table we see that the weight remains unchanged, except if the bottom label is an \textbf{R}, in which case it decreases by 2. 

                \begin{table}[H]
                \centering
                \scalebox{1.1}{
                \renewcommand{\arraystretch}{1.3} 
                \begin{tabular}{c| c c c c}  

 & \textbf{L} & \textbf{N} & \textbf{R} & \textbf{I}  \\ \hline
\textbf{R} & 1 & 2 & \textbf{2} & 0 \\
 \textbf{I} & 1 & 2 & \textbf{0} & 0

\end{tabular}}
\end{table}
    \end{itemize}
\end{proof}

\begin{cor} \label{corrr}
 Let $f \in F(n)$, and suppose that left-multiplying $f$ by $x_i^{-1}$, where $i\in \{1,2, \dots, n-1\}$ removes a caret in the top forest. Then $\ell(x^{-1}_i f) = \ell(f) \pm 1$. Specifically, $\ell(x^{-1}_i f) > \ell(f) $ if and only if the right space of $x_i^{-1}f$ has label $\rule{0pt}{24pt} \left[ \begin{matrix} \mathbf{R} \\ \mathbf{R} \end{matrix} \right]$.
\end{cor}

Combining the first parts of Propositions \ref{prop1} and \ref{prop3} and Proposition \ref{prop2} is enough to verify condition (2) of Lemma \ref{lemma}. Also, along the way we have acquired almost enough information to verify condition (3). We finalise this with Theorem \ref{cond3}.

\begin{thm} \label{cond3}
    Let $f \in F(n)$ be a non-identity element with forest diagram $\mathfrak{f}$. We can apply the following procedure to reduce the length of $f$  
    \begin{enumerate}

        \item If left-multiplication by $x_i$, for some $i\in \{1,2, \dots, n-1\}$, removes a caret from the bottom forest of $\mathfrak{f}$, then $\ell(x_i f) < \ell(f)$.
        
        \item Otherwise, if left-multiplication by $x_i^{-1}$, for some $i\in \{1,2, \dots, n-1\}$, removes a caret in the top forest of $\mathfrak{f}$, then either $\ell(x_i^{-1} f) < \ell(f)$, or  $\ell(x_0 f) < \ell(f)$ .

        \item If neither of the two previous conditions hold, then either $\ell(x_0 f) < \ell(f)$ or $\ell(x_0^{-1} f) < \ell(f)$.  
    \end{enumerate}
\end{thm}

\begin{proof}
Let us prove each of the three statements individually.
  \vspace{3mm}

    \textit{Statement 1:} \quad See Proposition \ref{prop2}.
      \vspace{3mm}

    \textit{Statement 2:} \quad By Corollary \ref{corrr}, if $\ell(x_i^{-1} f) > \ell(f)$ then the right space of $x_i^{-1} f$ has label $\left[ \begin{matrix} \mathbf{R} \\ \mathbf{R} \end{matrix} \right]$, which implies that the right space of $f$ must have one of the four following labels: $\rule{0pt}{24pt}\left[ \begin{matrix} \mathbf{R} \text{ or } \textbf{N} \\ \mathbf{R} \text{ or } \textbf{N} \end{matrix} \right]$, which all carry a weight of 2. However, the $\rule{0pt}{14pt}$ top label will become an \textbf{L} in $x_0 f$, and so now the associated weight will be 1. Therefore, $\ell(x_0 f) < \ell(f)$.

    \vspace{3mm}

    \textit{Statement 3: } \quad By Proposition \ref{prop1} $\ell(x_0 f) = \ell(f) \pm 1$. Let us examine the three cases that lead to $\ell(x_0 f) > \ell(f)$:
    \begin{enumerate}
        \item $x_0 f$ has larger support than $f$. This means that the top pointer of $\mathfrak{f}$ is at the right end of the support of $f$ and that the left space of $f$ has label $\rule{0pt}{24pt}\left[ \begin{matrix} \mathbf{L} \\ \mathbf{R} \end{matrix} \right]$, $\left[ \begin{matrix} \mathbf{L} \\ \mathbf{L} \end{matrix} \right]$ or $\left[ \begin{matrix} \mathbf{L} \\ \mathbf{I} \end{matrix} \right]$. In the first case, the left space of $f$ will not be in the su$\rule{0pt}{14pt}$pport of $x_0^{-1} f$, and so it will be unweighted. For the two other cases, notice that the top label will become an \textbf{R} and so the weight of the left space decreases. Hence, in all three cases we have $\ell(x_0^{-1} f) < \ell(f)$.

        \item The right space of $f$ has bottom label \textbf{L} and left-multiplication by $x_0$ does not remove this space from the support. Then the left space of $f$ can only have top label \textbf{L} (since it is by definition exterior) and by a similar argument the only two possibilities for bottom label are \textbf{L} or \textbf{I}. In both cases, the top label will become an \textbf{R} in $x_0^{-1} f$ and so the weights will decrease. Hence,  $\ell(x_0^{-1} f) < \ell(f)$.

        \item The right space of $f$ has label $\left[ \begin{matrix} \mathbf{R} \\ \mathbf{I} \end{matrix} \right]$. As the bottom label is \textbf{I} the space is enclosed by an $n$-caret from the bottom forest whose $\rule{0pt}{14pt}$ leftmost vertex is to the left of the top pointer (otherwise left-multiplication by $x_i$ would have removed a caret from the bottom forest). Also, this $n$-caret can only have children whose leftmost vertex is to the left of the top pointer. All of this implies that the left space of $f$ must have label $\rule{0pt}{24pt}\left[ \begin{matrix} \mathbf{L} \\ \mathbf{I} \end{matrix} \right]$. As before, the top label of the left space of $f$ will become an \textbf{R} in $x_0^{-1}f$ and so the associated weight will decrease from 1 to 0. Hence, $\ell(x_0^{-1} f) < \ell(f)$.  
    \end{enumerate}
\end{proof}

Note that, in Theorem \ref{cond3}, statements (1) and (2) are not mutually exclusive; this is only true for the $F(2)$ case. For example, for the element with the forest diagram pictured in Figure \ref{fig:counter_example}, note that left-multiplying by either $x^{-1}_1$ or $x_2$ would remove a caret from the top or bottom forests, respectively; however, only left-multiplying by $x_2$ would reduce the length. This is why statement (1) has ``priority" over statement (2).
\begin{figure}[H]
    \centering
    \includegraphics[width=0.25\linewidth]{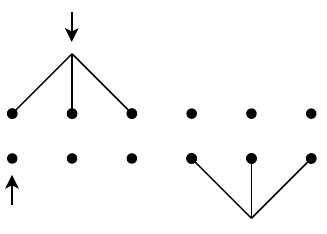}
    \caption{An example of why order matters in Theorem \ref{cond3}.}
    \label{fig:counter_example}
\end{figure}
To finish this section, let us give a corollary from the previous theorem.  
\begin{cor}\label{neeedeed}
    Let $f\in F(n)$ with reduced forest diagram $\mathfrak{f}$. Then there exists a minimum-length word $w$ for $f$ with the following properties:
    \begin{enumerate}
        \item Each instance of $x_i$ for any $i\in \{ 1, \dots, n-1\}$ creates a caret in the top forest of $\mathfrak{f}$.
        \item Each instance of $x^{-1}_i$ for any $i\in \{ 1, \dots, n-1\}$ creates a caret in the bottom forest of $\mathfrak{f}$.
    \end{enumerate}
\end{cor}

\section{Dead Ends and Deep Pockets} \label{final-section}

Dead end elements were first introduced by Bogopolskii, in \cite{bogopolskii1997infinite}.
\begin{defn}
    Let $G$ be a group with finite generating set $S$ and a length function $\ell: G \rightarrow \mathbb{N}$ with respect to $S$. An element $g \in G$ is a \emph{dead end} if $\ell(sg) \leq \ell(g)$ for all $s \in S \cup S^{-1}$, where $S^{-1}$ denotes the set of inverses of elements of $S$.
\end{defn}

\begin{exa}
Let $f \in F(3)$ be the element with the following forest diagram. 
    \begin{figure}[H]
        \centering
        \includegraphics[width=1\linewidth]{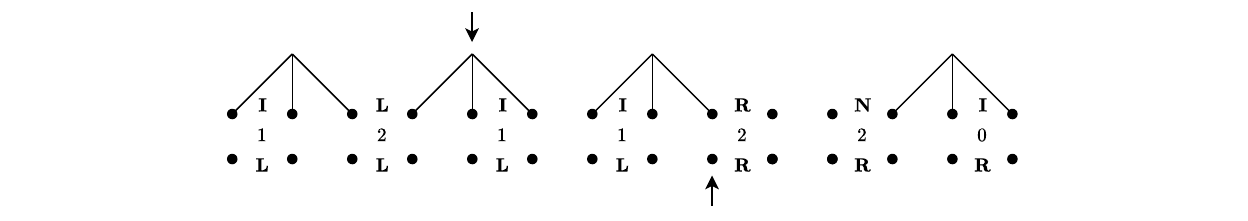}
        \vspace{-5mm}
        \caption{A typical dead end element $f \in F(3)$.}
        \label{fig:dead-end}
    \end{figure}

Note that, as the left space of $f$ has label $\left[ \begin{matrix} \mathbf{L} \\ \mathbf{L} \end{matrix} \right]$  left-multiplying by $x_0^{-1}$ will decrease the length,  by Corollary \ref{cor1}. Also, $\rule{0pt}{14pt} $ left-multiplying $f$ by $x_0$ will decrease the length since the right space of $f$ will now have top label \textbf{L}. Furthermore, since the right space of $f$ has label $ \rule{0pt}{24pt} \left[ \begin{matrix} \mathbf{R} \\ \mathbf{R} \end{matrix} \right]$, multiplying $f$ by $x_i$ for any $i \in \{1, \dots,n-1\}$ will decrease the length, by $\rule{0pt}{14pt}$ Proposition \ref{prop3}. Finally, left-multiplying $f$ by $x^{-1}_i$ for any $i\in \{1, \dots, n-1\}$ will decrease the length since a caret will be deleted and none of the weights will increase (this could only happen if the top label were \textbf{L}, which can not be the case). 
\end{exa}

We can characterise any dead end element in $F(n)$ with the following proposition. 
\begin{prop} \label{prop-dead-ends}
    Let $f \in F(n)$. Then $f$ is a dead end if and only if: 
    \begin{enumerate}
        \item Left-multiplying $f$ by $x^{-1}_i$ for any $i\in \{1,\dots,n-1\}$ will cancel a caret from the top forest, 
        \item the left space of $f$ has label $ \left[ \begin{matrix} \mathbf{L} \\ \mathbf{L} \end{matrix} \right]$,
        \item the right space of $f$ has label $ \left[ \begin{matrix} \mathbf{R} \\ \mathbf{R} \end{matrix} \right]$, and
        \item the right space of $x^{-1}_i f$ does not have label $ \left[ \begin{matrix} \mathbf{R} \\ \mathbf{R} \end{matrix} \right]$ for any $i\in \{1,\dots,n-1\}$.
    \end{enumerate}
\end{prop}
\begin{proof} 
It should be easy to see that the ``if'' direction follows trivially. To prove the ``only if" direction, let us assume  that $f$ is a dead end. 
\vspace{3mm}

  \textbf{Condition 1:}  as $f$ is a dead end we have that $\ell(x_i f) < \ell(f)$ for any $i\in \{1,\dots,n-1\}$, so the result follows by  the contrapositive of Corollary \ref{hmm}.
\vspace{3mm}

 \textbf{Condition 2:}    this is a direct consequence of Corollary \ref{cor1}, given that $\ell(x_0^{-1} f) < \ell(f)$.

  \vspace{3mm}

\textbf{Condition 3:}    now follows directly from Proposition \ref{prop3}, since $\ell(x_i f) < \ell(f)$.

  \vspace{3mm}
\textbf{Condition 4:}   now follows from Corollary \ref{corrr}, since $\ell(x^{-1}_if)< \ell(f)$ for any $i\in \{1,\dots,n-1\}$.
\end{proof}

Notice that there are several ways to meet condition 4. Also observe, that in the proof of Proposition \ref{prop-dead-ends}, we did not use the fact that $\ell(x_0f) < \ell(f)$. Therefore, we can conclude that $f \in F(n)$ is a dead end element without considering how left-multiplying by $x_0$ will affect the length. 

\begin{defn}
     Let $G$ be a group with finite generating set $S$ and length function $\ell: G \rightarrow \mathbb{N}$ with respect to $S$. We say a dead end element $g \in G$ has \emph{depth} $k$ if there exists a smallest $k \in \mathbb{N}$ such that 
     $$ \ell(s_1 \cdots s_{k+1} \cdot g) \leq \ell(g) $$
     for all $s_1, \dots, s_{k+1} \in S \cup S^{-1}$, where $S^{-1}$ denotes the set of inverses of elements of $S$.
\end{defn}

In \cite{wladis2008unusual}, using Fordham's method, Wladis proves that in $F(n)$ there exist no dead end elements with depth $k \geq 3$. That is, all dead end elements have depth two. We give an alternative proof of this fact.

\begin{prop}
  Let $f \in F(n)$ be a dead end element. Then $f$ has depth two. 
\end{prop}
\begin{proof}
    Let $f \in F(n)$ be a dead end element. By Proposition \ref{prop-dead-ends}, the right space of $f$ has label $ \left[ \begin{matrix} \mathbf{R} \\ \mathbf{R} \end{matrix} \right]$, which implies that the current tree  of $x_0 f$ and the $n-2$ trees to the right of it must be  $\rule{0pt}{14pt}$ trivial. Therefore, left-multiplying $x_0 f$ by $x^{-1}_i$ for any $i \in \{1, \dots,n-1\}$ will create a caret in the bottom forest, whose enclosed marked space will be labelled $ \rule{0pt}{24pt} \left[ \begin{matrix} \mathbf{R} \\ \mathbf{I} \end{matrix} \right]$ and hence carry a weight of 0. Thus, each time we left-multiply $x_0f$ by $x^{-1}_i$ for any $i \in \{1, \dots,n-1\}$, the length will increase by one. 

    Finally, since $f$ is a dead end we have that $\ell(x_0f)= \ell(f)-1$ and therefore $\ell(x_i x_j x_0f) = \ell(f)+1$, for any $i,j \in \{1, \dots,n-1\}$. That is, the depth of $f$ is two. 
\end{proof}

Note that, as opposed to the case for $F(2)$ where for a dead end element $f$ there is only one possibility of $g\in F(2)$ such that $\ell(g f) = \ell(f)+1$ (namely $g=x^{-2}_1x_0$), in $F(n)$ we have $(n-1)^2$ different ways of doing this.

\begin{exa}
    Let $f \in F(3)$ be the dead end element from Figure \ref{fig:dead-end}. The following are all the ways we can left-multiply $f$ by $g \in F(3)$ such that $\ell(gf)= \ell(f)+1$.
    
    \begin{figure}[H]
        \includegraphics[width=1\linewidth]{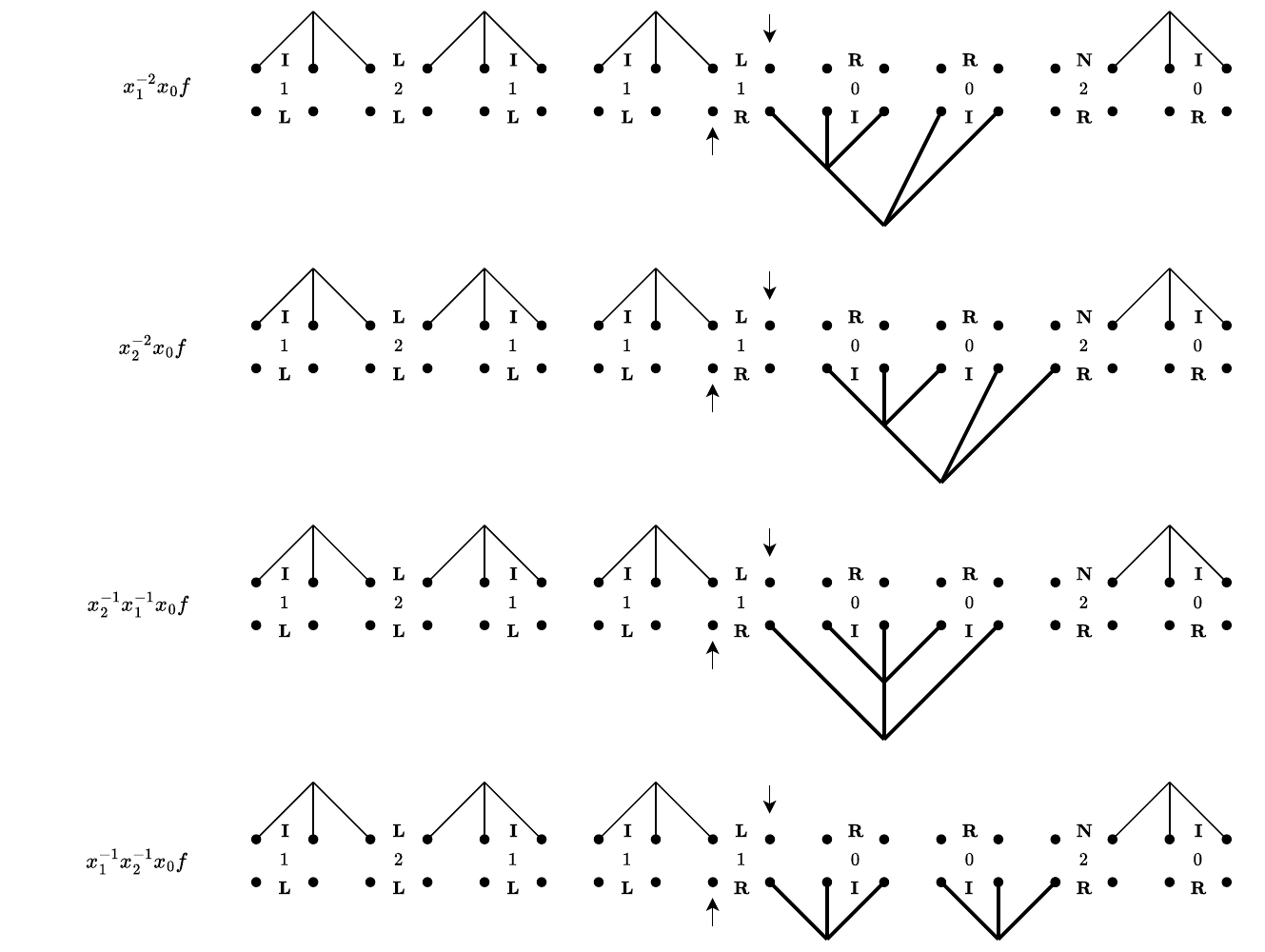}
    \end{figure}

\end{exa}

\bibliographystyle{alpha}
\bibliography{bibliography}

@book{burrillobook,
    author = {Burillo, José},
    title = {Introduction to {Thompson’s Group $F$}},
    publisher = {\url{https://web.mat.upc.edu/pep.burillo/F%20book.pdf}},
    note = {Accessed: 10th of October 2024}
}

@article{belk2005forest,
  title={{``Forest Diagrams for Elements of {Thompson's Group $F$}"}},
  author={Belk, James and Brown, Kenneth},
  journal={International Journal of Algebra and Computation},
  volume={15},
  number={05n06},
  pages={815--850},
  year={2005},
  publisher={World Scientific}
}

@article{fordham2009minimal,
  title={{``Minimal Length Elements of {Thompson’s groups $F (p)$}''}},
  author={Fordham, S Blake and Cleary, Sean},
  journal={Geometriae Dedicata},
  volume={141},
  pages={163--180},
  year={2009},
  publisher={Springer}
}

@article{BROWN198745,
title = {{``Finiteness Properties of Groups''}},
journal = {Journal of Pure and Applied Algebra},
volume = {44},
number = {1},
pages = {45-75},
year = {1987},
issn = {0022-4049},
doi = {https://doi.org/10.1016/0022-4049(87)90015-6},
url = {https://www.sciencedirect.com/science/article/pii/0022404987900156},
author = {Kenneth Brown}
}

@book{higman1974finitely,
  title={Finitely Presented Infinite Simple Groups},
  author={Higman, G},
  isbn={9780708103005},
  lccn={74077210},
  series={Notes in Pure Math},
  url={https://books.google.com.br/books?id=LPvuAAAAMAAJ},
  year={1974},
  publisher={Department of Pure Mathematics, Department of Mathematics, I.A.S., Australian National University}
}

@article{blake2003minimal,
  title={{``Minimal Length Elements of Thompson's Group $F$''}},
  author={Fordham, S Blake},
  journal={Geometriae Dedicata},
  volume={99},
  number={1},
  pages={179--220},
  year={2003},
  publisher={Springer}
}

@article{bogopolskii1997infinite,
  title={{``Infinite Commensurable Hyperbolic Groups are Bi-Lipschitz Equivalent''}},
  author={Bogopolskii, Oleg },
  journal={Algebra and Logic},
  volume={36},
  number={3},
  pages={155--163},
  year={1997},
  publisher={Springer}
}

@article{cleary2004combinatorial,
  title={{``Combinatorial Properties of Thompson’s Group $F$''}},
  author={Cleary, Sean and Taback, Jennifer},
  journal={Transactions of the American Mathematical Society},
  volume={356},
  number={7},
  pages={2825--2849},
  year={2004}
}

@article{wladis2008unusual,
  title={{``Unusual Geodesics in Generalizations of Thompson’s Group $F$"}},
  author={Wladis, Claire},
  journal={Illinois Journal of Mathematics},
  volume={53},
  number={2},
  pages={483--514},
  year={2009},
  publisher={Duke University Press}
}

@article{cleary2003thompson,
  title={{``Thompson's Group $F$ is Not Almost Convex''}},
  author={Cleary, Sean and Taback, Jennifer},
  journal={Journal of Algebra},
  volume={270},
  number={1},
  pages={133--149},
  year={2003},
  publisher={Elsevier}
}

@incollection{BelkBux2005,
  author    = {James Belk and Kai-Uwe Bux},
  title     = {{``Thompson’s Group $F$ is Maximally Nonconvex"}},
  booktitle = {Geometric Methods in Group Theory},
  series    = {Contemporary Mathematics},
  volume    = {372},
  publisher = {American Mathematical Society},
  pages     = {131--146},
  year      = {2005},
  mrnumber  = {2139683},
  url       = {https://www.ams.org/conm/372/}
}

@article{elder2010counting,
  title={{``Counting Elements and Geodesics in Thompson's Group $F$"}},
  author={Elder, Murray and Fusy, {\'E}ric and Rechnitzer, Andrew},
  journal={Journal of Algebra},
  volume={324},
  number={1},
  pages={102--121},
  year={2010},
  publisher={Elsevier}
}

@phdthesis{francesco,
    author = {Francesco Mattuci} ,
    title = {{``Algorithms and Classification in Groups of Piecewise-Linear Homeomorphisms"}},
    school = {Cornell University} ,
    year = {2008}
}

\end{document}